\documentclass{amsart}
\usepackage{graphicx}
\newtheorem{theorem}{Theorem}[section]
\newtheorem{lemma}[theorem]{Lemma}

\newtheorem{corollary}[theorem]{Corollary}

\newtheorem{proposition}[theorem]{Proposition}

\theoremstyle{definition}
\newtheorem{definition}[theorem]{Definition}
\newtheorem{remark}[theorem]{Remark}

\begin{document}

\title{Dehn surgery and Seifert surface system}

\author{Makoto Ozawa}
\address{Department of Natural Sciences, Faculty of Arts and Sciences, Komazawa University, 1-23-1 Komazawa, Setagaya-ku, Tokyo, 154-8525, Japan}
\email{w3c@komazawa-u.ac.jp}
\thanks{The first author is partially supported by Grant-in-Aid for Scientific Research (C) (No. 23540105 and No. 26400097), The Ministry of Education, Culture, Sports, Science and Technology, Japan}
\author{Koya Shimokawa}
\address{Department of Mathematics, Saitama University, Saitama 338-8570, Japan}
\email{kshimoka@rimath.saitama-u.ac.jp}
\thanks{The second author is partially supported by Grant-in-Aid for Scientific Research (C) (No. 25400080), The Ministry of Education, Culture, Sports, Science and Technology, Japan}

\subjclass[2010]{57M25}

\keywords{link, Dehn surgery, Seifert surface}

\begin{abstract}
For a compact connected 3-submanifold with connected boundary in the 3-sphere, we relate the existence of a Seifert surface system for a surface with a Dehn surgery along a null-homologous link.
As its corollary, we obtain a refinement of the Fox's re-embedding theorem.
\end{abstract}

\maketitle

\section{Introduction}

\begin{definition}
Let $M$ be a compact connected 3-manifold with connected boundary of genus $g$.
A {\em spanning surface system} $\{F_i\}$ for $M$ is a set satisfying the following.
\begin{enumerate}
\item $\{F_i\}$ is a set of disjoint orientable surfaces properly embedded in $M$.
\item $\{\partial F_i\}$ is a set of $g$ disjoint loops $C_1, \ldots, C_g$ which do not separate $\partial M$.
\end{enumerate}
A spanning surface system $\{F_i\}$ for $M$ is {\em completely disjoint} if $\{F_i\}$ is a set of $g$ disjoint orientable surfaces.
\end{definition}

\begin{remark}\label{Lei}
We remark by \cite[Corollary 1.4]{L} that if $M$ is a handlebody and $\{F_i\}$ is a completely disjoint spanning surface system for $M$, then there exists a meridian disk system $\{D_i\}$ for $M$ such that $\partial F_i=\partial D_i$ for $i=1,\ldots,g$.
\end{remark}

By a homological argument, we have the following proposition.

\begin{proposition}\label{existence}
Any compact connected 3-submanifold with connected boundary in $S^3$ admits a spanning surface system.
\end{proposition}

However, a compact connected 3-submanifold with connected boundary in $S^3$ does not always admit a completely disjoint spanning surface system.

\begin{remark}\label{Lambert}
We remark by \cite{Lam} that there exists a compact connected 3-submanifold with connected boundary of genus 2 in $S^3$ which does not admit a completely disjoint spanning surface system.
\end{remark}

\begin{definition}
Let $S$ be a genus $g>0$ closed surface in $S^3$, and put $S^3=V\cup_S V'$.
A {\em Seifert surface system} $(\{F_i\}, \{F_i'\})$ for $S$ is a pair of sets satisfying the following.
\begin{enumerate}
\item $\{F_i\}$ (resp. $\{F_i'\}$) is a spanning surface system for $V$ (resp. $V'$).
\item $|C_i\cap C_j'|=\delta_{ij}$ for $i, j=1,\ldots, g$, where $\{\partial F_i\}=\{C_1, \ldots, C_g\}$ and $\{\partial F_i'\}=\{C_1', \ldots, C_g'\}$
\end{enumerate}
A Seifert surface system $(\{F_i\},\ \{F_i'\})$ for $S$ is {\em completely disjoint} if $\{F_i\}$ and $\{F_i'\}$ are completely disjoint.
\end{definition}

\begin{definition}
Let $L=L_1\cup \cdots \cup L_n$ be a link in $S^3$.
Following \cite{MOS}, we say that $L$ is a {\em reflexive link} if the 3-sphere can be obtained by a non-trivial Dehn surgery along $L$.
In particular, if the surgery slope $\gamma_i$ for $L_i$ is $1/n_i$ for some integer $n_i$ $(i=1,\ldots, n)$, then we call the Dehn surgery a {\em $1/\Bbb{Z}$-Dehn surgery}.

Suppose that $L$ is contained in a compact 3-submanifold $M$ in $S^3$.
We say that $L$ is {\em null-homologous} in $M$ if $[L]=0$ in $H_1(M;\Bbb{Z})$, and that $L$ is {\em completely null-homologous} in $M$ if $[L_i]=0$ in $H_1(M;\Bbb{Z})$ for $i=1,\ldots,n$.
\end{definition}

\begin{theorem}\label{main}
Let $M$ be a compact connected 3-submanifold with connected boundary in $S^3$.
Then the followings hold.
\begin{enumerate}
\item There exists a null-homologous link $L$ in $M$, which is reflexive in $S^3$, such that a handlebody can be obtained from $M$ by a $1/\Bbb{Z}$-Dehn surgery along $L$.
\item $M$ admits a completely disjoint spanning surface system if and only if there exists a completely null-homologous link $L$ in $M$, which is reflexive in $S^3$, such that a handlebody can be obtained from $M$ by a $1/\Bbb{Z}$-Dehn surgery along $L$.
\end{enumerate}
\end{theorem}

We remark that in Theorem \ref{main} (2), we can take a completely null-homologous reflexive link $L$ so that it is disjoint from the completely disjoint spanning surface system.

\begin{remark}\label{equivalent}
Let $M$ be a compact connected 3-submanifold of $S^3$ with connected boundary of genus $g$.
Let $f:M\to X$ be a map onto a genus $g$ handlebody $X$.
We say that $f$ is a {\em boundary preserving map} of $M$ onto $X$ if $f$ is continuous and $f|_{\partial M}$ is a homeomorphism onto $\partial X$.
We say that $M$ is {\em retractable} if $M$ can be retracted onto a wedge of $g$ simple closed curves.
If such a wedge can be chosen to be in $\partial M$, then $M$ is called {\em boundary retractable}.
Set $G=\pi_1(M)$ and define $G_1=[G,G]$, $G_{n+1}=[G_n,G]$, $G_{\omega}=\bigcap_n G_n$.
Then the following conditions are equivalent.
\begin{enumerate}
\item $M$ admits a completely disjoint spanning surface system.
\item There exists a boundary preserving map from $M$ onto a handlebody.
\item $M$ is boundary retractable.
\item The natural map $\pi_1(\partial M)\to G/G_{\omega}$ is an epimorphism.
\end{enumerate}

The equivalence between (1) and (2) was shown in \cite[Theorem 2]{Lam}.
The equivalence between (2) and (3) was shown in \cite[Theorem 3]{JM}.
The equivalence between (3) and (4) was shown in \cite[Theorem 2, 3]{J}.
\end{remark}

Let $M$ be a compact connected 3-submanifold of $S^3$.
By Proposition \ref{existence}, each component of the exterior of $M$ admits a spanning surface system.
If we adapt Theorem \ref{main} (1) to the exterior of $M$, then we obtain a refinement of the Fox's re-embedding theorem as the following corollary.

\begin{corollary}[\cite{F,ST,N}]\label{Fox}
Every compact connected 3-submanifold $M$ of $S^3$ can be re-embedded in $S^3$ so that the exterior of the image of $M$ is a union of handlebodies.
\end{corollary}

\begin{remark}
In relation with Remark \ref{equivalent}, there is an another equivalence condition.
Let $M$ be a compact connected 3-submanifold of $S^3$ with connected boundary of genus $2$.
By Corollary \ref{Fox}, there exists a re-embedding of $M$ so that its exterior is a genus $2$ handlebody $V$.
A handcuff graph shaped spine $\Gamma$ of $V$ is a {\em boundary spine} if its constituent link $L_{\Gamma}$ is a boundary link that admits a pair of disjoint Seifert surfaces whose interiors are contained in $S^3-\Gamma$.
A handlebody $V$ is {\em $(3)_S$-knotted} if it does not admit any boundary spine.
Then it was shown in \cite[Theorem 3.10]{BF} that $M$ admits a completely disjoint spanning surface system if and only if $H$ is not $(3)_S$-knotted.
\end{remark}

Theorem \ref{main} also deduces the following corollary.
It follows from (2) of the next corollary that every closed surfaces with completely disjoint Seifert surface systems can be related by $1/\Bbb{Z}$-Dehn surgeries along completely disjoint null-homologous reflexive links. 

\begin{corollary}\label{closed_surface}
Let $S$ be a closed surface in $S^3$ which separates $S^3$ into 3-submanifolds $M$ and $M'$
Then the followings hold.
\begin{enumerate}
\item There exist null-homologous links $L$ and $L'$ in $M$ and $M'$, which are reflexive in $S^3$, such that handlebodies can be obtained from $M$ and $M'$ by $1/\Bbb{Z}$-Dehn surgeries along $L$ and $L'$.
\item $S$ admits a completely disjoint Seifert surface system if and only if there exist completely null-homologous links $L$ and $L'$ in $M$, which are reflexive in $S^3$, and $M'$ such that handlebodies can be obtained from $M$ and $M'$ by $1/\Bbb{Z}$-Dehn surgeries along $L$ and $L'$.
\end{enumerate}
\end{corollary}

By Corollary \ref{closed_surface} (1), we can obtain a Seifert surface system from a meridian-longitude disk system for the handlebodies by tubing along the null-homologous links.

\begin{corollary}\label{existence2}
Any closed surface in $S^3$ admits a Seifert surface system.
\end{corollary}

Let $M$ be a 3-manifold.
Let $L\subset M$ be a submanifold with or without boundary.
When $L$ is 1 or 2-dimensional, we write $E(L)=M-int\,N(L)$.
When $L$ is of 3-dimension, we write $E(L)=M-int\,N(L)$.

\section{Proof}

Let $V$ be a genus $g$ handlebody in $S^3$, and $\{D_i\}$ be a meridian disk system for $V$.
Since $V-\bigcup_i int N(D_i)$ is a 3-ball, there exists a spine $\Gamma$ of $V$ such that:
\begin{enumerate}
\item $\Gamma$ consists of $g$ loops $l_1,\ldots,l_g$ and $g$ arcs $\gamma_1,\ldots,\gamma_g$ connecting $l_i$ to a point $x$.
\item The point $x$ is the center of the 3-ball $V-\bigcup_i int N(D_i)$, and which is homeomorphic to $N(x\cup \gamma_1\cup\cdots\cup\gamma_g)$.
\item Each loop $l_i$ is dual to $D_i$.
\end{enumerate}
We call this spine $\Gamma$ a {\em $g$-handcuff graph shaped spine} for $V$ with respect to $\{D_i\}$.

\begin{figure}[htbp]
	\begin{center}
	\includegraphics[trim=0mm 0mm 0mm 0mm, width=.6\linewidth]{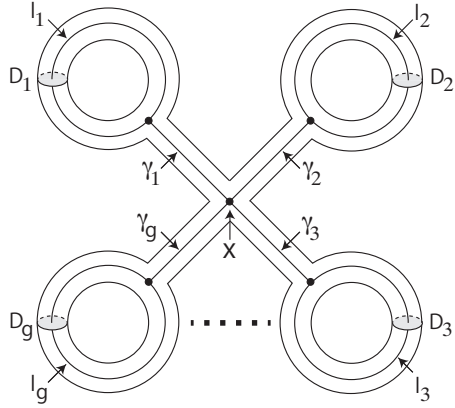}\\
	
	\end{center}
	\caption{A $g$-handcuff graph shaped spine for $V$ with respect to $\{D_i\}$}
	\label{spine}
\end{figure}

Next, let $\{F_i\}$ be a set of orientable surfaces with boundary and without closed component.
We say that $\{F_i\}$ is a {\em Seifert surface system} for $\Gamma$ if $(\bigcup_i F_i) \cap \Gamma=\bigcup_i \partial F_i=\bigcup_i l_i$.

\begin{lemma}\label{handcuff_bound}
Any $g$-handcuff graph shaped spine in $S^3$ admits a Seifert surface system.
\end{lemma}

\begin{proof}
We take a regular diagram of $\Gamma$ such that $x\cup \gamma_1\cup\cdots\cup\gamma_g$ has no crossing.
Then we apply the Seifert's algorithm (\cite{S}) to loops $l_1\cup\cdots\cup l_g$ with arbitrary orientations, and obtain a Seifert surfaces $\{F_i'\}$ for the loops.
\end{proof}

The following lemma states that from any meridian disk system for a handlebody, we can obtain a Seifert surface system for the boundary of the bandlebody.

\begin{lemma}\label{extension}
Let $V$ be a genus $g$ handlebody in $S^3$ with a meridian disk system $\{D_i\}$.
Then there exists a spanning surface system $\{F_i\}$ for $E(V)$ such that $(\{D_i\},\{F_i\})$ is a Seifert surface system for $\partial V$.
\end{lemma}

\begin{proof}
Let $\Gamma$ be a $g$-handcuff graph shaped spine $\Gamma$ for $V$ with respect to $\{D_i\}$.
By Lemma \ref{handcuff_bound}, $\Gamma$ admits a Seifert surface system $\{F_i'\}$.
The restriction $\{F_i'\}$ to $E(V)$ gives a spanning surface system, say $\{F_i\}$, for $E(V)$ such that $(\{D_i\},\{F_i\})$ is a Seifert surface system for $\partial V$.
\end{proof}

Let $\Gamma$ be a $g$-handcuff graph shaped spine with a Seifert surface system $\{F_i\}$.
We call the operation of (1) in Figure \ref{crossing} a {\em band-crossing change} of $\{F_i\}$, and the operation of (2) in Figure \ref{crossing} a {\em full-twist} of $\{F_i\}$.
We remark that these operations can be obtained by a $1/\Bbb{Z}$-Dehn surgery along trivial links in the complement of $\{F_i\}$.

\begin{figure}[htbp]
	\begin{center}
	\begin{tabular}{cc}
	\includegraphics[trim=0mm 0mm 0mm 0mm, width=.5\linewidth]{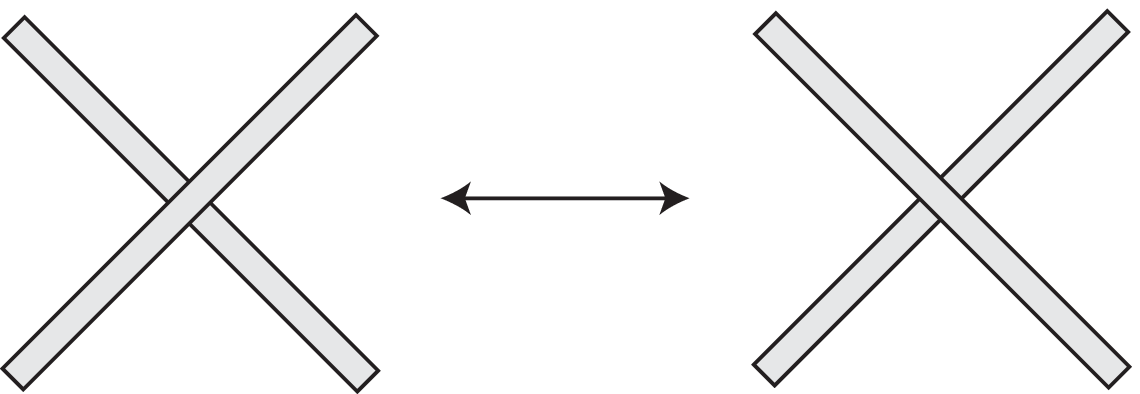}&
	\includegraphics[trim=0mm 0mm 0mm 0mm, width=.15\linewidth]{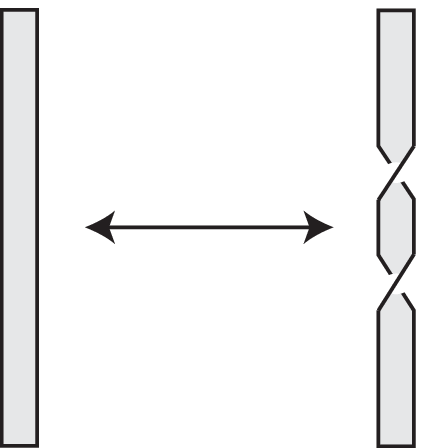}\\
	(1) a band-crossing change & (2) a full-twist\\
	\end{tabular}
	\end{center}
	\caption{A band-crossing change and a full-twist of $\{F_i\}$}
	\label{crossing}
\end{figure}

\begin{lemma}\label{unknotted}
Any $g$-handcuff graph shaped spine $\Gamma$ with a Seifert surface system $\{F_i\}$ can be unknotted by band-crossing changes and full-twists of $\{F_i\}$.
\end{lemma}

\begin{figure}[htbp]
	\begin{center}
	\includegraphics[trim=0mm 0mm 0mm 0mm, width=.6\linewidth]{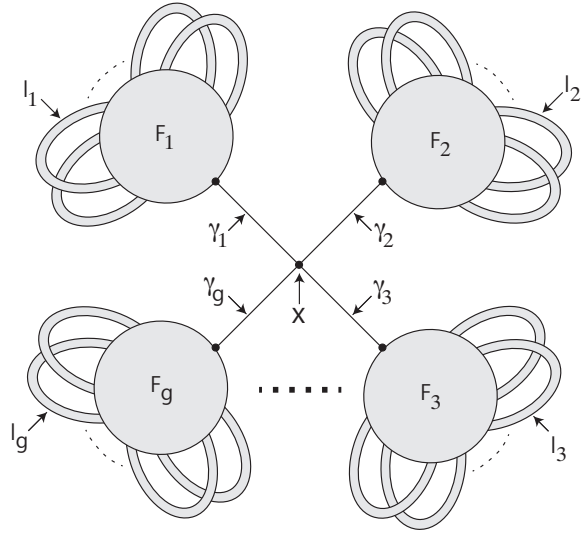}\\
	
	\end{center}
	\caption{A $g$-handcuff graph shaped spine $\Gamma$ with a Seifert surface system $\{F_i\}$, which is a ``standard planar form''}
	\label{standard}
\end{figure}

\begin{proof}
It is observed that $\Gamma$ with $\{F_i\}$ can be transformed to be a ``standard planar form'' (cf. \cite{K}) by the following operations.
\begin{enumerate}
\item a band-crossing change of $\{F_i\}$
\item a full-twists of $\{F_i\}$
\item a crossing change between $\{F_i\}$ and $\{\gamma_i\}$
\item a crossing change among $\{\gamma_i\}$
\end{enumerate}
However, the operations (3) and (4) can be exchanged for the operation (1).
If $\Gamma$ with $\{F_i\}$ has a standard planar form, then $\Gamma$ is unknotted and this completes the proof.
We remark that in a standard planar form, $l_1\cup\cdots\cup l_g$ is the trivial link.
\end{proof}

\begin{lemma}\label{core}
Let $L$ be a reflexive link in $S^3$ which is contained in a compact 3-submanifold $M$ in $S^3$.
Suppose that $L$ in null-homologous (resp. completely null-homologous) in $M$.
Then the core link $L^*$ in the 3-submanifold $M'$ obtained by a $1/\Bbb{Z}$-Dehn surgery along $L$ is also null-homologous (resp. completely
null-homologous) in $M'$.
\end{lemma}

\begin{proof}
Suppose that $L$ in null-homologous (resp. completely null-homologous) in $M$.
Then $L$ bounds a Seifert surface $F$ (resp. completely disjoint Seifert surface) in $M$.
Put $F^*=F\cap E(L)$.
By a $1/\Bbb{Z}$-Dehn surgery, the meridian of the core link $L^*$ intersects each component of $\partial F^*$ in one point.
This shows that $F^*$ can be extended to a a Seifert surface (resp. completely disjoint Seifert surface) for $L^*$ in $M'$.
Thus $L^*$ is also null-homologous (resp. completely
null-homologous) in $M'$.
\end{proof}

\begin{lemma}\label{handlebody}
Let $V$ be a handlebody in $S^3$.
Then the followings hold.
\begin{enumerate}
\item $\partial V$ admits a Seifert surface system if and only if there exists a null-homologous link $L$ in $E(V)$, which is reflexive in $S^3$, such that a handlebody can be obtained from $E(V)$ by a $1/\Bbb{Z}$-Dehn surgery along $L$.
\item $\partial V$ admits a completely disjoint Seifert surface system if and only if there exists a completely null-homologous link $L$ in $E(V)$, which is reflexive in $S^3$, such that a handlebody can be obtained from $E(V)$ by a $1/\Bbb{Z}$-Dehn surgery along $L$.
\end{enumerate}
\end{lemma}


\begin{proof}
(1) Suppose that there exists a null-homologous reflexive link $L$ in $E(V)$ such that a handlebody, say $W$, can be obtained from $E(V)$ by a $1/\Bbb{Z}$-Dehn surgery along $L$.
There exists a meridian disk system $\{D_i\}$ for $W$.
Since $L$ is null-homologous in $E(V)$, by Lemma \ref{core}, the core link $L^*$ is also null-homologous in $W$.
Therefore, we can obtain a Seifert surface system for $E(V)$ by tubing $\{D_i\}$ along $L^*$.

Conversely, suppose that $\partial V$ admits a Seifert surface system $(\{F_i\}, \{F_i'\})$, where $\{F_i\}$ and $\{F_i'\}$ are spanning surface systems for $V$ and $E(V)$.
By Remark \ref{Lei}, we may assume that each $F_i$ is a disk.
Then there exists a $g$-handcuff graph shaped spine $\Gamma$ and $F_i'$ can be extended to a Seifert surface system for $\Gamma$.
By Lemma \ref{unknotted}, $\Gamma$ with $\{F_i'\}$ can be unknotted, hence $E(V)$ is a handlebody, by band-crossing changes and full-twists of $\{F_i'\}$.
This operations can be obtained by a $1/\Bbb{Z}$-Dehn surgery along a trivial link $L$ in $\Gamma\cup \bigcup_i F_i'$.
Since $L$ is contained in $E(V)-\bigcup_i F_i'$, $L$ is a null-homologous link in $E(V)$.
Hence we obtain a null-homologous reflexive link $L$ in $E(V)$ such that a handlebody can be obtained from $E(V)$ by a $1/\Bbb{Z}$-Dehn surgery along $L$.

(2) This can be proved by the argument similar to (1).
\end{proof}

\begin{lemma}\label{heegaard}
Let $S$ be a Heegaard surface in $S^3$ which decomposes $S^3$ into two handlebodies $V$ and $V'$.
Let $\{D_i\}$ be a meridian disk system for $V$.
Then there exist a null-homologous reflexive link $L'$ in $V'$ which yields a handlebody $V''$ by a $1/\Bbb{Z}$-Dehn surgery on $L'$ and a meridian disk system $\{D_i''\}$ for $V''$ such that $(\{D_i\},\{D_i''\})$ is a completely disjoint Seifert surface system for $S$ in $V\cup V''$.
\end{lemma}

\begin{proof}[Proof of Lemma \ref{heegaard}]
We take a $g$-handcuff graph shaped spine $\Gamma$ of $V$ with respect to $D_i$.
Since $\Gamma$ can be unknotted by crossing changes, there exists a null-homologous reflexive link $L'$ in $V'$ such that after a $1/\Bbb{Z}$-Dehn surgery along $L'$, every loops of $\Gamma$ bound mutually disjoint disks.
Therefore, a handlebody $V''$ obtained from $V'$ by a $1/\Bbb{Z}$-Dehn surgery along $L'$ admits a meridian disk system $\{D_i''\}$ so that $(\{D_i\},\{D_i''\})$ is a completely disjoint Seifert surface system for $S$.
\end{proof}

\begin{proof}[Proof of Theorem \ref{main}]
(1) We prove by an induction on the genus $g(\partial M)$.
Since the 3-sphere does not contain an incompressible closed surface, there exists a compressing disk $D$ for $\partial M$ in $S^3$.
We divide the proof into two cases.

\begin{description}
\item[Case 1] $D\subset M$
\item[Case 2] $D\subset E(M)$
\end{description}

In Case 1, put $M'=M-{int}\,N(D)$.
By the assumption of the induction, there exists a null-homologous reflexive link $L'$ in $M'$ such that handlebodies can be obtained from $M'$ by a $1/\Bbb{Z}$-Dehn surgery along $L'$.
This proves the theorem since $M$ is obtained by adding 1-handle $N(D)$ to $M'$.

In Case 2, we take a maximal compression body $W$ for $\partial M$ in $E(M)$ \cite{B}.
If $W$ is a handlebody (i.e. $W=E(M)$), then the theorem follows Lemma \ref{extension} and Lemma \ref{handlebody} (1).
Otherwise as $g(\partial W)<g(\partial M)$, by the assumption of the induction, there exists a null-homologous reflexive link $L'$ in each component of $E(M)-{int}\,W$ such that handlebodies can be obtained from the component by a $1/\Bbb{Z}$-Dehn surgery along $L'$.
After these $1/\Bbb{Z}$-Dehn surgery, $E(M)$ is a handlebody.
Therefore, again by Lemma \ref{extension} and Lemma \ref{handlebody} (1), there exists a null-homologous reflexive link $L$ in $M$ such that a handlebody can be obtained from $M$ by a $1/\Bbb{Z}$-Dehn surgery along $L$.
Finally, we recover the previous $1/\Bbb{Z}$-Dehn surgery on each component of $E(M)-{int}\,W$ to obtain the original $E(M)$.

(2) Suppose that there exists a completely null-homologous reflexive link $L$ in $M$ such that a handlebody can be obtained from $M$ by a $1/\Bbb{Z}$-Dehn surgery along $L$.
There exists a meridian disk system $\{D_i\}$ for the resultant handlebody.
Since $L$ is completely null-homologous in $M$, we can obtain a completely disjoint spanning surface system for $M$ by tubing $\{D_i\}$ along $L$.

Conversely, suppose that $M$ admits a completely disjoint spanning surface system $\{F_i\}$.
In the following 3 steps, we convert $M$ and $E(M)$ into two handlebodies $V$ and $V''$ so that $(V,V'')$ admits a meridian disk system $(\{D_i\},\{D_i''\})$ with $\partial F_i=\partial D_i$.

Step 1: By (1) of this theorem, there exists a null-homologous link reflexive $L$ in $E(M)$ such that a handlebody can be obtained from $E(M)$ by a $1/\Bbb{Z}$-Dehn surgery along $L$.
Let $V'$ the resultant handlebody obtained from $E(M)$ and note that $M\cup V'$ is again the 3-sphere.

Step 2: We note that there exists a degree one map from $M$ to a handlebody $V$ which sends each $F_i$ to a meridian disk $D_i$ of $V$ and preserves the boundary of $M$ (cf. {\cite[Theorem 2]{Lam}}, {\cite[Theorem 5]{Hem}}).
We naturally extend this degree one map to a degree one map $\phi:S^3=M\cup V' \to X=V\cup V'$ as follows.
\begin{enumerate}
\item $V'$ is contained in $X$ by an inclusion.
\item Each $F_i$ is sent to a meridian disk $D_i$ of the handlebody $\phi(M)=V$.
\item The remnant $M-\bigcup int\,N(F_i)$ is sent to the 3-ball $V-\bigcup int\,N(D_i)$.
\end{enumerate}
Since $\phi_*:\pi_1(S^3)\to \pi_1(X)$ is surjective \cite[Lemma 15.12]{H}, $X$ is homeomorphic to $S^3$ \cite{P1,P2,MT}.

Step 3: By Lemma \ref{heegaard}, there exists a null-homologous reflexive link $L'$ in $V'$ and a meridian disk system $\{D_i''\}$ for a handlebody $V''$ obtained from $V'$ by a $1/\Bbb{Z}$-Dehn surgery along $L'$ such that $(\{D_i\}, \{D_i''\})$ is a completely disjoint Seifert surface system for $(V, V'')$.


Since the degree one map $\phi$ is a boundary preserving map by the condition (1), $(\{F_i\}, \{D_i''\})$ is a completely disjoint Seifert surface system for $(M, V'')$.
By Lemma \ref{handlebody} (2), there exists a completely null-homologous reflexive link $L_0$ in $M$ such that a handlebody can be obtained from $M$ by a $1/\Bbb{Z}$-Dehn surgery along $L_0$.
Moreover, by the proof of Lemma \ref{handlebody} (2), we can take $L_0$ so that $L_0\cap \bigcup F_i=\emptyset$.
Thus the completely disjoint spanning surface system $\{F_i\}$ is contained in the resultant handlebody $V_0$ obtained from $M$ by a $1/\Bbb{Z}$-Dehn surgery along $L_0$.
\end{proof}

\bigskip


\bibliographystyle{amsplain}

\end{document}